\documentclass[12pt]{article}

\usepackage[a4paper, top=3cm, left=2.5cm, right=2.5cm, bottom=3cm]{geometry}
\usepackage{amsmath}
\usepackage{amsthm, pb-diagram}
\usepackage{amssymb,comment}
\usepackage{amsfonts,graphicx,color}
\usepackage{enumerate, fancyhdr, dsfont}
\usepackage[normalem]{ulem}

    \newcommand{\Mwf}{\mathcal{M}}
    \newcommand{\Nwf}{\mathcal{N}}

    \newcommand{\Aor}{\mathbb{A}}
    \newcommand{\Bor}{\mathds{B}}
    \newcommand{\Cor}{\mathds{C}}
    \newcommand{\Dor}{\mathds{D}}
    \newcommand{\Eor}{\mathds{E}}
    
    \newcommand{\Lor}{\mathds{L}}
    \newcommand{\Loc}{\mathds{LOC}}
    
    \newcommand{\Mor}{\mathds{M}}
    \newcommand{\Por}{\mathds{P}}

    \newcommand{\Ior}{\mathds{I}}
    
    \newcommand{\Vor}{\mathds{V}}
    
    \newcommand{\Sor}{\mathds{S}}

    \newcommand{\SNwf}{\mathcal{SN}}
    \newcommand{\SMwf}{\mathcal{SM}}

    \newcommand{\la}{\langle}
    \newcommand{\ra}{\rangle}

\newcommand{\Lb}{\mathbf{Lb}}

\newcommand{\hgt}{\mathrm{ht}}

\title{Cohen real or random real: effect on strong measure zero sets and strongly meager sets}

\author{Miguel A. Cardona}

\date{Institute of Discrete Mathematics and Geometry.\\
Faculty of Mathematics and Geoinformation.\\
\addvspace{\medskipamount}
TU Wien.}

\begin{document}

\AtEndDocument{\bigskip{\footnotesize%
  \textrm{Institute of Discrete Mathematics and Geometry} \par 
  \textrm{Faculty of Mathematics and Geoinformation}\par
  \textrm{TU Wien}\par
  \textrm{Wiedner Hauptstrasse 8--10/104 A--1040 Wien}\par
  \textsc{Austria}\par
  \textit{E-mail address}: \texttt{miguel.montoya@tuwien.ac.at} \par
  \textit{URL}:\, https://www.researchgate.net/profile/Miguel\_Cardona\_Montoya \par
}}

\makeatletter
\def\@roman#1{\romannumeral #1}
\makeatother

\newcounter{enuAlph}
\renewcommand{\theenuAlph}{\Alph{enuAlph}}

\renewcommand{\theequation}{\thesection.\arabic{equation}}

\theoremstyle{plain}
  \newtheorem{theorem}{Theorem}[section]
  \newtheorem{corollary}[theorem]{Corollary}
  \newtheorem{lemma}[theorem]{Lemma}
  \newtheorem{mainlemma}[theorem]{Main Lemma}
  \newtheorem{maintheorem}[enuAlph]{Main Theorem}
  \newtheorem{prop}[theorem]{Proposition}
  \newtheorem{clm}[theorem]{Claim}
  \newtheorem{exer}[theorem]{Exercise}
  \newtheorem{question}[theorem]{Question}
  \newtheorem{problem}[theorem]{Problem}
  \newtheorem{conjecture}[theorem]{Conjecture}
  \newtheorem*{thm}{Theorem}
  \newtheorem{teorema}[enuAlph]{Theorem}
  \newtheorem*{corolario}{Corollary}
\theoremstyle{definition}
  \newtheorem{definition}[theorem]{Definition}
  \newtheorem{example}[theorem]{Example}
  \newtheorem{remark}[theorem]{Remark}
  \newtheorem{notation}[theorem]{Notation}
  \newtheorem{context}[theorem]{Context}
\newtheorem{fact}[theorem]{Fact}
  \newtheorem*{defi}{Definition}
  \newtheorem*{acknowledgements}{Acknowledgements}
  \newtheorem*{main-theorem}{Main Theorem}

\maketitle


\begin{abstract}
We show that the set  of  the  ground-model  reals has strong measure zero (is strongly meager) after adding a single Cohen real (random real). As consequence we prove that the set  of  the  ground-model  reals has strong measure zero after adding a single Hechler real.
\end{abstract}
\maketitle

\section{Introduction}
Let $\Nwf$ be the $\sigma$-ideal of measure zero subsets of $2^\omega$, and  let $\Mwf$ be the $\sigma$-ideal of meager sets in $2^\omega$. More concretely $X\in\Mwf$ if there is some sequence $\la F_n:n<\omega\ra$ such that $X= \bigcup_{n<\omega}F_n$ and $\textrm{int}(\textrm{cl}(F_n))=\emptyset$. Let $\Cor$ and $\Bor$ be the Cohen algebra and random algebra respectively, let $\Dor$ be the Hechler forcing, let $\Lor$ be the Laver forcing, let $\Mor$ be Mathias forcing, let $\Vor$ be Silver forcing and let $\Sor$ be Sacks forcing. 

\begin{definition} For each $\sigma\in(2^{<\omega})^\omega$ define $\hgt\in\omega^\omega$ by $\hgt_\sigma(n):=|\sigma(n)|$.

Say that \emph{$X\subseteq 2^\omega$ has strong measure zero} ($X\in\SNwf$) if, for each function $f\in\omega^\omega$ there is some $\sigma\in(2^{<\omega})^\omega$ with $\hgt_\sigma=f$ such that  $X\subseteq \bigcup_{n<\omega}[\sigma(n)]$. 

It is clear that $\SNwf\subseteq\Nwf$.
\end{definition}

Galvin, Mycielski and Solovay \cite{GaMS} gave a very important description of the strong measure zero sets.

\begin{theorem}[{\cite{GaMS}}]
The following are equivalent:
\begin{itemize}
    \item[(1)] $X\in\SNwf$,
    \item[(2)] for every set $F\in\Mwf$, there is some $x\in2^\omega$ such that $(x+X)\cap F=\emptyset$. 
\end{itemize}
\end{theorem}

Using this characterization, we consider the following objects.

\begin{definition}
We say that  \emph{$X\subseteq2^\omega$ is strongly meager} ($X\in\SMwf$) if, for each $N\in\Nwf$, there is $x\in2^\omega$ such that $(X+x)\cap N=\emptyset$. 

It is clear that $\SMwf\subseteq\Mwf$.
\end{definition}

Kunen \cite{Kunen} proved that after adding a single Cohen real (random real) the set of the ground-model reals becomes null (meager). More presicely,
\begin{theorem}[{\cite{Kunen}}]\label{cohran} If $c$ and $r$ are a Cohen real and a random real over $V$ respectively, then 
\begin{itemize} 
    \item[(i)] $V[c]\models 2^\omega\cap V\in \Nwf$ and $2^\omega\cap V\notin\Mwf$. In particular, $V[c]\models 2^\omega\cap V\notin \SMwf$.
    \item[(ii)] $V[r]\models 2^\omega\cap V\in \Mwf$ and $2^\omega\cap V\notin\Nwf$. In particular, $V[r]\models 2^\omega\cap V\notin \SNwf$.
\end{itemize}
\end{theorem}

Motivated by Theorem \ref{cohran}. in this paper we prove that the set of the ground-model reals has strong measure zero after adding a single Cohen real. This was mentioned by Laver \cite{Laver} (without proof), afterwards, Goldstern sketched this in \cite{cohenSN}. We also prove that the set of the ground-model reals is strongly meager after adding a single random real. This was sketched in \cite{randomSM}. The author present a complete proof of these results with some slight variations associated with his perpective. 

\section{Main result}

This section is dedicated to prove the following main result.

\begin{teorema}\label{cohenrandom}
If $c$ and $r$ are a Cohen real and a random real over $V$ respectively, then 
\begin{itemize}
    \item[(i)] $V[c]\models2^\omega\cap V\in\SNwf$.
    \item[(ii)] $V[r]\models2^\omega\cap V\in\SMwf$.
\end{itemize}
\end{teorema}

\begin{proof}
\begin{itemize} 
\item[(i)] 
Enumerate $2^{<\omega}:=\{r_n:n<\omega\}$. For each $f\in\omega^{\omega}$ and $F\in\omega^\omega$ define \[B^c_{f,F}:=\bigcup_{n\in\omega}[r_{c(F(n))}{}^{\smallfrown}\la0,\ldots,0\ra]\]
where for each $n$, the length of $\la0,\ldots,0\ra$ is the greatest between $f(n)-|r_{c_{F(n)}}|$ and 0. Note that $B^c_{f,F}$ is coded in $V[c]$. It is enough to prove that, for any $\Cor$-name $\dot{f}$ in $\omega^\omega$ there is a function $F\in\omega^\omega$ such that $\Vdash_{\Cor} 2^\omega\cap V\subseteq B^{\dot{c}}_{\dot{f},F}$. 

In $V$ define a function $F_p\in\omega^\omega$ for each $p\in\Cor$ by \[F_p(m):=\min\Big\{k\in\omega:\exists q\in\Cor(|q|=k\wedge q\leq p\wedge\exists l<\omega(q\Vdash \dot{f}(m)=l))\Big\},\]
Choose $F\in\omega^\omega$ such that $F_p\leq^* F$ for all $p\in\Cor$. It remains to check that $\Vdash_{\Cor} 2^\omega\cap V\subseteq B^{\dot{c}}_{\dot{f},F}$. To do this, let $p$ be an arbitrary condition in $\Cor$. Choose $n<\omega$ such that $F_p(m)\leq F(m)$ for all $m\geq n$.
Now choose $q\in\Cor$ with $|q|=F_p(n)$ and $l<\omega$ such that $q$ extends $p$ and $q\Vdash \dot{f}(n)=l$. Let $x\in2^\omega\cap V$. Find $i<\omega$ such that $r_i:=x\upharpoonright l$. Define a condition $q^*\in\Cor$ such that $|q^*|=F(n)+1$, $q^*\Vdash \dot{c}(F(n))=i$ and $q^*\leq q$. 

Then, $q^*\Vdash x\in[r_{\dot{c}(F(n))}]\subseteq B^{\dot{c}}_{\dot{f},F}$ (this contention holds because $|r_{\dot{c}(F(n))}|=l=\dot{f}(n)$). 

\item[(ii)] For an increasing function $f\in\omega^\omega$ and a function $x\in2^\omega$ define $x_f\in2^\omega$ as $x_f(n):=x(f(n))$ for $n\in\omega$. Let $A$ be a Borel set in $V[r]\cap\Nwf$. In $V$ find a Borel null set such that $B\subseteq 2^\omega\times2^\omega$ and $A=B_r$. Since $B$ has measure zero, choose sequences $s_n, t_n\in2^{<\omega}$ with $|s_n|=|t_n|$ such that \[B\subseteq \bigcap_{m<\omega}\bigcup_{n\geq m}[s_n]\times[t_n]\textrm{\ and\ }\sum_{n=1}^{\infty}2^{-2|s_n|}<\infty.\]
Find an increasing function $f\in\omega^\omega$ by induction on $n$ such that 
\begin{itemize}
    \item[(a)] $j\leq f(n)\to |s_j|<f(n+1)$.
    \item[(b)] $\sum_{j\geq f(n)}\Lb([s_j]\times[t_j])\leq\frac{\Lb([s_n]\times[t_n])}{2^{n+2}}$
\end{itemize}
From (a) and (b) it follows that 
\begin{equation*}\label{eq1}
\begin{split}
(\star)\,\,\sum\limits_{f(n)\leq j<f(n+1)}\frac{2^{|f^{-1}[|s_j|]|}}{2^{2|s_j|}}&\leq\sum\limits_{f(n)\leq j< f(n+1)}\frac{2^{n+2}}{2^{2|s_j|}}\\
&\leq\Lb([s_n]\times[t_n]).
\end{split}
\end{equation*}
We first show that, for each $z\in V\cap
2^\omega$, \[\bigg\{x:\la x,x_{f}+z\ra\in\bigcap_{m<\omega}\bigcup_{n\geq m}[s_n]\times[t_n]\bigg\}\] has measure zero. To this end, let \[H_n^z:=\Big\{x:\la x,x_{f}\ra\in[s_n]\times[z\upharpoonright\!|t_n|+t_n]\Big\}\]
Then we have 
\[\bigg\{x:\la x,x_{f}+z\ra\in\bigcap_{m<\omega}\bigcup_{n\geq m}[s_n]\times[t_n]\Bigg\}=\bigcap_{m<\omega}\bigcup_{n\geq m} H_n^z\]
It remains to prove that $\bigcap_{m<\omega}\bigcup_{n\geq m} H_n^z$ has measure zero.
\begin{clm}\label{clm}
\[\Lb(H_n^z)\leq\frac{2^{|f^{-1}(|s_n|)|}}{2^{2|s_n|}}\]
\end{clm} 
\begin{proof}
Note that $H_n^z=[s_n]\cap[(z\upharpoonright\!|t_n|+t_n)\circ f^{-1}]$. Let $t':=(z\upharpoonright\!|t_n|+t_n)\circ f^{-1}$. Then 
$H_n^z=[s_n]\cap[(z\upharpoonright\!|t_n|+t_n)\circ f^{-1}]=\emptyset$ when $s_n$ and $t'$ are incompatible. Otherwise, 
\begin{align*}
     H_n^z&=[s_n\cup((z\upharpoonright\!|t_n|+t_n)\circ f^{-1})]\\
     &=[s_n\cup t']
\end{align*}
Hence, 
\begin{align*}
    \Lb\Big([s_n\cup t']\Big)&=2^{-|s_n\cup t'|}\\ &=2^{-|s_n|-|\{f(n):n<|t_n|\wedge f(n)\geq |s_n|\}|}\\
    &\leq2^{-|s_n|-|t_n|+|f^{-1}[|s_n|]}\\
    &=\frac{2^{|f^{-1}(|s_n|)|}}{2^{2|s_n|}}.
\end{align*}
This  ends  the  proof  of Claim \ref{clm}.
\end{proof}
We  continue  the  proof  of (ii). It follows that $\bigcap_{m<\omega}\bigcup_{n\geq m} H_n^z$ has measure zero by the Claim \ref{clm} and $(\star)$. In $V[r]$, since $r$ is a random real over $V$, $\la r,r_f+z\ra\not\in B$, which means that  $r_f+z\not \in A$. Therefore $(2^\omega\cap V)+A\neq2^\omega$ in $V[r]$.
\qedhere
\end{itemize}
\end{proof}

As a consequence of Theorem \ref{cohenrandom}, we get that the set of the ground-model reals has strong measure zero after adding a single Hechler real.

\begin{corollary}\label{hechler}
If $d$ is a Hechler real, then $V[d]\models2^\omega\cap V\in\SNwf$.
\end{corollary} 

Palumbo \cite{palu} proved that $\Dor\ast\Cor\equiv\Dor$, that is, $V[d'][c]=V[d]$ for some $\Dor$-generic real $d'$ over $V$ and a Cohen real $c$ over $V[d']$. By Theorem \ref{cohenrandom}, $V[d'][c]\models 2^\omega\cap V[d']\in\SNwf$, in particular $V[d'][c]\models2^\omega\cap V\in\SNwf$. Then $V[d]\models2^\omega\cap V\in\SNwf$.

The next result appears implicit in \cite{JMS}. 

\begin{theorem}
If $G$ and $G'$ are a $\Vor$-generic over $V$ and a $\Sor$-generic over $V$ respectively, then
\begin{itemize}
    \item[(a)] $V[G]\models2^\omega\cap V\notin\Nwf\cup\Mwf$, in particular, $V[G]\models2^\omega\cap V\notin\SNwf\cup\SMwf$.
    \item[(b)] $V[G']\models2^\omega\cap V\notin\Nwf\cup\Mwf$, in particular, $V[G']\models2^\omega\cap V\notin\SNwf\cup\SMwf$.
\end{itemize}
\end{theorem}

Miller \cite{Miller} introduced the infinitely  often equal real forcing $\Ior$ to prove that some combinatorial properties of measure and category of the real line are consistent. He also proved that the set of ground-model reals does not become meager (strongly null) after adding a single infinitely often equal real, in particular, the ground-model real does not become strongly meager. To summarize, 

\begin{theorem}
If $G$ is $\Ior$-generic over $V$, then
\begin{itemize}
    \item[(i)] $V[G]\models2^\omega\cap V\notin\SNwf$, and 
    \item[(ii)] $V[G]\models2^\omega\cap V\notin\SMwf$.
\end{itemize}
\end{theorem}

We finish this section with results related to the Laver property.

\begin{theorem}[{\cite{bartJudah},\cite[Theorem  8.5.20]{BJ}}]
Assume that $\Por$ has the Laver property. Then $\Vdash_{\Por}2^\omega\cap V\notin\SMwf$.
\end{theorem}

As a corollary we get

\begin{corollary}
If $G$ and $G'$ are $\Mor$-generic over $V$ and $\Lor$-generic over $V$´respectively, then 
\begin{itemize}
\item[(i)] $V[G]\models2^\omega\cap V\notin\SMwf$.
\item[(ii)] $V[G']\models2^\omega\cap V\notin\SMwf$.
\end{itemize}
\end{corollary}

On the other hand, Laver \cite{Laver} proved that adding an $\Mor$-generic over the ground-model $V$ forces all uncountable sets of reals in $V$ to not have strong measure
zero in the extension, that is, $V[G]\models2^\omega\cap V\notin\SNwf$. 

It is known that the set of the ground-model reals does not have measure zero after adding a $\Lor$-generic over $V$, that is, $V[G]\models2^\omega\cap V\notin\Nwf$, in particular $V[G]\models2^\omega\cap V\notin\SNwf$.

\section*{Open problems}

Miller \cite{Miller} proved that, if $c$ is a Cohen real over $V$ and $r$ is a random real over $V[c]$, then $V[c][r]\models 2^\omega\cap V[r]\notin\Mwf$, in particular $V[c][r]\models 2^\omega\cap V[r]\notin\SMwf$. Afterwards, Cicho\'n and Palikowski \cite{cicpal} proved that, if $r$ is a random real over $V$ and $c$ is a Cohen over $V[r]$, then $V[r][c]\models 2^\omega\cap V[c]\in\Nwf$. Later Palikowski \cite{paw86} proved that 
\begin{itemize}
    \item[(i)] If $r$ is a random real over $V$ and $c$ is a Cohen over $V[r]$, then \[V[r][c]\models 2^\omega\cap V[c]\notin\Mwf.\] In particular, $V[r][c]\models 2^\omega\cap V[c]\notin\SMwf$.
    \item[(ii)] If $c$ is a Cohen real over $V$ and $r$ is a random over $V[c]$, then \[V[c][r]\models 2^\omega\cap V[r]\in\Nwf.\]
\end{itemize}

We ask the following problems.

\begin{question}
If $c$ is a Cohen real over $V$ and $r$ is a random over $V[c]$, does \[V[c][r]\models 2^\omega\cap V[r]\in\SNwf?\]   
\end{question}
    
\begin{question}
If $r$ is a random real over $V$ and $c$ is a Cohen over $V[r]$, does \[V[r][c]\models 2^\omega\cap V[c]\in\SNwf?\]
\end{question}
 
In Corollary \ref{hechler} it was proved that the ground-model real become strongly null after adding a single Hechler real, but it is still open the following question. 
 
\begin{question}
If $d$ is a Hechler real over $V$, does $V[d]\models2^\omega\cap V\in\SMwf$?
\end{question}
 
It is known that 
\begin{itemize}
    \item[(a)] $\Vdash2^\omega\cap V\in\Nwf$.
    \item[(b)] $\Vdash2^\omega\cap V\in\Mwf$.
\end{itemize}
for the following posets:
\begin{itemize}
    \item[(1)] The eventually different real forcing $\Eor$.
    \item[(2)] The localization forcing $\Loc$.
    \item[(3)] Amoeba forcing $\Aor$.
\end{itemize}
 
It is natural to ask:

\begin{question}
For the posets in the list above do we have 
\begin{itemize}
    \item[(i)] $\Vdash2^\omega\cap V\in\SNwf$?
    \item[(ii)] $\Vdash2^\omega\cap V\in\SMwf$?
\end{itemize}
\end{question}

\subsection*{Acknowledgments} This work was supported by the Austrian Science Fund (FWF) P30666 and the author is a recipent of a DOC Fellowship of the Austrian Academy of Sciences at the Institute of Discrete Mathematics and Geometry, TU Wien. 

This paper was developed for the conference proceedings corresponding to the Set
Theory Workshop that Professor Daisuke Ikegami organized in November 2019. The
author is very thankful to Professor Ikegami for letting him participate in such wonderful
workshop.

The author thanks  Dr.\ Diego Mej\'ia for his valuable comments while working on this paper. The author also thanks  Dr.\ Martin Goldstern for reading this work and for his remark on  references and grammar corrections.

{\small
\bibliography{left}
\bibliographystyle{alpha}
}

\end{document}